\documentclass[12pt]{amsart} \usepackage{amsmath,centernot, amssymb,amsthm,mathrsfs}
\usepackage[hidelinks]{hyperref}
\usepackage{physics}

\newtheorem{theorem}{Theorem}
\newtheorem{lemma}[theorem]{Lemma}
\newtheorem{proposition}[theorem]{Proposition}

\numberwithin{equation}{section}
\numberwithin{theorem}{section}

\theoremstyle{definition}

\newtheorem{definition}[theorem]{Definition}
\newtheorem{observation}[theorem]{Observation}

\begin{document}

\title[Bohr neighborhoods]{Bohr neighborhoods in three-fold difference sets}

\author{John T. Griesmer}
\address{Department of Applied Mathematics and Statistics\\ Colorado School of Mines, Golden, Colorado}
\email{jtgriesmer@gmail.com}


\begin{abstract}
  Answering a question of Hegyv{\'a}ri and Ruzsa, we show that if $A$ is a set of integers having positive upper Banach density, then the set $A+A-A:=\{a+b-c: a, b, c\in A\}$ contains Bohr neighborhoods of many elements of $A$, where the radius and dimension of the Bohr neighborhood depend only on $d^{*}(A)$.
\end{abstract}

\maketitle

\section{Introduction}\label{secIntroduction}  For a real number $t$, $\|t\|$ denotes distance to the nearest integer. If $k\in \mathbb N$, $\eta>0$, and $s_{1},\dots, s_{k}$ are real numbers, then a \emph{Bohr-$(k,\eta)$ set} is a set of  integers of the form $\{n: \max_{i\in \{1,\dots,k\}} \|s_{i}n\|<\eta\}$. A Bohr-$(k,\eta)$ neighborhood of $n$ is a set of the form $n+U$, where $U$ is a Bohr-$(k,\eta)$ set.  The parameters $\eta$ and $k$ are called the \emph{radius} and \emph{dimension} of $U$, respectively.

Let $d^{*}(A)$ denote the upper Banach density of a set $A\subseteq \mathbb Z$. Theorem 2.2 of \cite{HegyvariRuzsa} says that if $d^{*}(A)>0$, then $A+A-A$ is a Bohr neighborhood of many $a\in A$.  The proof therein does not specify parameters $k, \eta$ for the Bohr neighborhood in terms of $d^{*}(A)$, and Section 3 of \cite{HegyvariRuzsa} asks for a proof which makes those parameters effective. Our main result is the following theorem, which provides the requested effective bounds.

\begin{theorem}\label{thmMainSpecial}
\begin{enumerate}
  \item[1.] Let $A\subseteq \mathbb Z$ have $d^{*}(A)>0$.  There are constants $k\in \mathbb N$, $\eta>0$, depending only on $d^{*}(A)$, such that $A+A-A-a$ contains a Bohr-$(k,\eta)$ set for some $a\in A$.

\smallskip

\item[2.]  For all $\varepsilon>0$, there are constants $k\in \mathbb N, \eta>0$, depending only on $d^{*}(A)$ and $\varepsilon$, such that there is a set $A'\subseteq A$ satisfying $d^{*}(A\setminus A')<\varepsilon$, and $A+A-A-a'$ contains a Bohr-$(k,\eta)$ set for all $a'\in A'$.
    \end{enumerate}
\end{theorem}
The estimate $d^{*}(A\setminus A')<\varepsilon$ cannot be improved to $d^{*}(A\setminus A')=0$, but we omit examples to this effect, as the constructions are tedious.

Our proof of Theorem \ref{thmMainSpecial} generalizes without modification to the setting of countable abelian groups, so we work in that context.  The next section introduces some terminology and notation for countable abelian groups, and states Theorem \ref{thmMain}, the natural generalization of Theorem \ref{thmMainSpecial} to that setting.  The proof of Theorem \ref{thmMain} is supplied by Proposition \ref{propComplete}, which reduces the study of $A+A-A$ to an analogous problem for compact abelian groups, solved in Section \ref{secCompactAbelian}. The proof of Proposition \ref{propComplete} is carried out in Section \ref{secErgodic} using ergodic theoretic methods similar to those of \cite{GrIsr,GrAdv}.

It would be interesting to find a shorter, more elementary proof of Theorem \ref{thmMainSpecial}.
\section{Countable abelian groups}\label{secCountableAbelian}
Let $\Gamma$ be a countable  abelian group.  If $A, B\subseteq \Gamma$, $\gamma\in \Gamma$,  write $A+B$ for $\{a+b:a\in A, b\in B\}$, $A-B$ for $\{a-b: a\in A, b\in B\}$, and $\gamma+A$ for $\{\gamma+a: a\in A\}$.

\subsection{Upper Banach density} A F{\o}lner sequence for $\Gamma$ is a sequence  $(\Phi_{n})_{n\in \mathbb N}$ of finite subsets of $\Gamma$ such that $\lim_{n\to \infty} \frac{|\Phi_{n}\triangle (\Phi_{n}+\gamma)|}{|\Phi_{n}|}=0$ for every $\gamma\in \Gamma$.  It is well known that every countable abelian group admits a F{\o}lner sequence.  If $\mathbf \Phi = (\Phi_{n})_{n\in \mathbb N}$ is a F{\o}lner sequence and $A\subseteq \Gamma$, let $\overline{d}_{\mathbf \Phi}(A):= \limsup_{n\to \infty} \frac{|A\cap \Phi_{n}|}{|\Phi_{n}|}$ be the \emph{upper density of $A$ with respect to $\mathbf \Phi$}.  Write $d_{\mathbf \Phi}(A)$ for $\overline{d}_{\mathbf \Phi}(A)$ if the limit exists.  For $A\subseteq \Gamma$, the \emph{upper Banach density of $A$} is  $d^{*}(A):=\sup\{d_{\mathbf \Phi}(A): \mathbf \Phi \text{ is a F{\o}lner sequence}\}$.

\subsection{Bohr neighborhoods}  Let $\mathbb T$ denote the group $\mathbb R/\mathbb Z$.  For $t\in \mathbb R$, recall that $\|t\|$ is the distance from $t$ to the nearest integer. For $x\in \mathbb T$, $\|x\|$ is defined to be $\|\tilde{x}\|$, where $\tilde{x}\in \mathbb R$ satisfies $\tilde{x}+\mathbb Z = x$.

If $S$ is a finite set of homomorphisms $\rho: \Gamma\to \mathbb T$, $|S|=k$, and $\eta>0$, then $\{\gamma \in \Gamma: \|\rho(\gamma)\|<\eta \text{ for all } \rho \in S\}$ is a called a \emph{Bohr-$(k,\eta)$ set}.  A \emph{Bohr-$(k,\eta)$ neighborhood} of $a\in \Gamma$ is a set of the form $a+U$, where $U$ is a Bohr-$(k,\eta)$ set.  These definitions of ``Bohr neighborhood" and ``Bohr set" agree with the definitions in Section \ref{secIntroduction}, as the maps $n\mapsto s_{i}n+\mathbb Z$ are homomorphisms from $\mathbb Z$ to $\mathbb T$.

Let $\mathcal S^{1}$ denote the group of complex numbers of modulus $1$, with the group operation of multiplication.  The groups $\mathbb T$ and $\mathcal S^{1}$ are isomorphic, via the isomorphism $e_{1}:\mathbb T\to \mathcal S_{1}$, $e_{1}(t):=\exp(2\pi i t)$.  This leads to the following observation.
\begin{observation}\label{obsBohrEquiv}
If $S$ is a finite set of homomorphisms $\rho: \Gamma\to \mathcal S^{1}$, $|S|=k$, and $\eta>0$, then the set $\{\gamma \in \Gamma: |\rho(\gamma)-1|<\eta \text{ for all } \rho\in S\}$ contains a Bohr-$(k,\eta/(2\pi))$ set.
\end{observation}

\begin{theorem}\label{thmMain}
Let $\varepsilon, \delta>0$, and let $\Gamma$ be a countable abelian group.
\begin{enumerate}
  \item[1.] There are constants $k\in \mathbb N$, $\eta>0$, depending only on $\delta$, such that if $A\subseteq \Gamma$,  has $d^{*}(A)\geq \delta$, then $A+A-A-a$ contains a Bohr-$(k,\eta)$ set for some $a\in A$.

\item[2.]  There are constants $k'\in \mathbb N, \eta'>0$, depending only on $\delta$ and $\varepsilon$, such that if $A\subseteq \Gamma$ has $d^{*}(A)\geq \delta$, there is a set $A'\subseteq A$ satisfying $d^{*}(A\setminus A')<\varepsilon$, and $A+A-A-a'$ contains a Bohr-$(k',\eta')$ set for all $a'\in A'$.
    \end{enumerate}
The constants $k,k',\eta, \eta'$ do not depend on $\Gamma$.
\end{theorem}

Part 2 follows immediately from Part 1, while Part 1 will be derived from Proposition \ref{propComplete} and proved in Section \ref{secReduction}.

\begin{proof}[Proof of Part 2.]
Assume Part 1 holds, and assume, to get a contradiction, that Part 2 fails. Given a set $A\subseteq \Gamma$, $k\in \mathbb N$, $\eta>0$, let
 \[E(A,k,\eta):=\{a\in A: A+A-A-a \text{ does not contain a Bohr-$(k,\eta)$ set}\}.\]
    Then there are $\delta>0$, $\varepsilon>0$ such that for all $k\in \mathbb N$, $\eta>0$, there exists $A$ having $d^{*}(A)>\delta$, while $d^{*}(E(A,k,\eta))>\varepsilon$.  Setting $E:=E(A,k,\eta)$, we have an $\varepsilon>0$ such that for all $k\in \mathbb N, \eta>0$, there is a set $E$ satisfying $d^{*}(E)>\varepsilon$ and for all $e\in E$, $E+E-E-e$ does not contain a Bohr-$(k,\eta)$ set.  This contradicts Part 1.
\end{proof}

\section{Compact abelian groups}\label{secCompactAbelian}

\subsection{Bohr neighborhoods in topological abelian groups}

\begin{definition} If $Z$ is a topological abelian group, $S$ is a set of \emph{continuous} homomorphisms $\rho: Z\to \mathbb T$, $|S|=k$, and $\eta>0$, then $\{z: \|\rho(z)\|<\eta \text{ for all } \rho \in S\}$ is a \emph{Bohr-$(k,\eta)$ set}.
 A Bohr-$(k,\eta)$ neighborhood of $z\in Z$ is a set of the form $z+U$, where $U$ is a Bohr-$(k,\eta)$ set.
\end{definition}
  If $Z$ is a topological abelian group, $\widehat{Z}$ denotes the character group of $Z$, meaning the group of continuous homomorphisms $\chi: Z\to \mathcal S^{1}$ with the group operation of pointwise multiplication.

\begin{observation}\label{obsBohrEquiv2} Following Observation \ref{obsBohrEquiv}, we see that if $S\subseteq \widehat{Z}$, $|S|=k$, and $\eta>0$, then the set $\{z: |\chi(z)-1|< \eta \text{ for all } \chi \in S\}$ contains a Bohr-$(k,\eta/(2\pi))$ set.
\end{observation}

\begin{lemma}\label{lemCompactToCountable}
  Let $\Gamma$ be a countable abelian group and $Z$ a topological abelian group.  If $\rho:\Gamma \to Z$ is a homomorphism and $U\subseteq Z$ contains a Bohr-$(k,\eta)$ set, then $\rho^{-1}(U)$ contains a Bohr-$(k,\eta)$ set in $\Gamma$.
\end{lemma}

\begin{proof}
  Let $S$ be a set of $k$ continuous homomorphisms $\psi:Z\to \mathbb T$ such that $U$ contains $\{z\in Z: \|\psi(z)\|< \eta \text{ for all } \psi\in S\}$.  Then $\rho^{-1}(U)$ contains $\{\gamma\in \Gamma: \|\psi \circ \rho(\gamma)\|<\eta \text{ for all } \psi \in S\}$, which is a Bohr-$(k,\eta)$ set.
\end{proof}

\begin{lemma}\label{lemBohrPoly}
  Let $Z$ be a topological abelian group and $S$ a finite subset of $\widehat{Z}$.  If $p:Z\to \mathbb C$, $p:=\sum_{\chi \in S} c_{\chi}\chi$, where $|c_{\chi}|\leq 1$ for all $\chi$, and $a\in Z$ satisfies $\Re p(a)\geq c$, then $\{x\in Z: \Re p(x)>0\}$ contains a Bohr-$(k,\eta)$ neighborhood of $a$, where $k$ and $\eta$ depend only on $|S|$ and $c$.  In fact we can use $k=|S|$ and $\eta=\frac{c}{2\pi|S|}$.
\end{lemma}

\begin{proof}
  Let $U:=\{x: |\chi(x)-1|< \frac{c}{|S|} \text{ for all } \chi \in S\}$, so that $U$ is a Bohr-$(|S|, c/(2\pi|S|))$ set, by Observation \ref{obsBohrEquiv2}.  Then for $x\in U$,
\begin{align*}
  |p(a+x)-p(a)| &\leq \sum_{\chi \in S} |\chi(a+x)-\chi(a)| \\
  &= \sum_{\chi\in S} |\chi(a)(\chi(x)-1)| \\
  &= \sum_{\chi \in S} |\chi(x)-1| \\
  &<c,
\end{align*}
  so $\Re p(a+x)>0$ for all $x\in U$, and $a+U \subseteq \{x: \Re p(x)>0\}$.
\end{proof}

The following lemma may be proved similarly.

\begin{lemma}\label{lemBohrIsOpen}
  Suppose $U$ is a Bohr-$(k,\eta)$ neighborhood of $a\in Z$.  Then there is a neighborhood $V$ of $a$ (in the topology of $Z$) such that $U$ is a Bohr-$(k,\eta/2)$ neighborhood of $z$ for all $z\in V$.
\end{lemma}

\subsection{Fourier identities}\label{secFourier}  We summarize some of the basic facts and definitions from harmonic analysis on compact abelian groups, available in standard references such as  \cite{Rudin}.

Let $Z$ be a compact abelian group with Haar measure $m$, normalized so that $m(Z)=1$.  For $\chi\in \widehat{Z}$ and $f\in L^{2}(m)$, $\hat{f}(\chi):= \int f\cdot \overline{\chi} \,dm$, and the function $\hat{f}:\widehat Z\to \mathbb C$ is the \emph{Fourier transform} of $f$.
For $f, g\in L^{1}(m)$, define the convolution $f*g$ by
\[
f*g(z):=\int f(z-t)g(t) \,dm(t).
 \]
Let $f, g\in L^{2}(m)$.  Then

\begin{align}
\label{eqnPlancherel} \int f\cdot \bar{g} \,dm &= \sum_{\chi \in \widehat Z} \hat{f}(\chi) \cdot \overline{\hat g(\chi)}\\
\label{eqnIsometry} \int |f|^{2} \,dm &= \sum_{\chi \in \widehat{Z}} |\hat{f}(\chi)|^{2}\\
\label{eqnConvolution} \widehat{f*g}&=\hat{f}\cdot \hat{g}.
\end{align}
Let $g_{-}$ be the function defined by $g_{-}(z)=g(-z)$.  Then
\begin{equation}\label{eqnMinus}
  \widehat{g_{-}}=\overline{\hat{g}}.
\end{equation}
Convolution is associative, so the three-fold convolution $h:=f*g*g_{-}$ is well-defined.  Furthermore, when $f, g\in L^{2}(m)$, $h$ is continuous and its Fourier series $\sum_{\chi \in \widehat{Z}} \hat{h}(\chi)\cdot \chi$ converges uniformly to $h$.


\begin{lemma}\label{lemConvolution}
Let $\delta>0$.  There exist $k\in \mathbb N$, $\eta>0$, depending only on $\delta$, such that if $Z$ is a compact abelian group with Haar measure $m$, and $f, g: Z\to [0,1]$ are measurable functions having $\int f \,dm, \int g \,dm \geq \delta$, then $\{z\in Z: f*g*g_{-}(z)>0\}$ contains a Bohr-$(k,\eta)$ neighborhood of  some $a\in Z$ having $f(a)>0$.
\end{lemma}

%
%

\begin{proof}
  Without loss of generality, we assume $\int f \,dm= \int g \,dm$, so let $f, g: Z\to [0,1]$ have $\int f \,dm = \int g \,dm = \delta$.

Let
  $h:= f*g*g_{-}$.  Then $h: Z\to [0,1]$, $h$ is continuous, and
    $\hat{h}=\hat{f}\cdot |\hat{g}|^{2}$, by Equations (\ref{eqnConvolution}) and (\ref{eqnMinus}).  Consequently $\hat{h}\in l^{1}(\widehat{Z})$.
Let $S_{1}:=\{\chi\in \widehat{Z}: |\hat{f}(\chi)|\geq \frac{1}{4}\delta^{3}\}$, $S_{2}:=\widehat{Z}\setminus S_{1}$.  We claim that
\begin{equation}\label{eqnS1Bound}
  |S_{1}| \leq 16\delta^{-5}
  \end{equation}
  and
\begin{equation}\label{eqnhBound}
  h(a) \geq \delta^{4} \text{ for some } a \text{ having } f(a)>0.
\end{equation}
We postpone the proofs of inequalities (\ref{eqnS1Bound}) and (\ref{eqnhBound}) and now prove the conclusion of the lemma.  Write the Fourier series of $h$ as $h(x) = p(x)+r(x)$, where
\begin{align*}
  p(x):= \sum_{\chi \in S_{1}} \hat{h}(\chi)\chi(x),\ \ r(x):= \sum_{\chi \in S_{2}} \hat{h}(\chi)\chi(x).
\end{align*}
Both series converge uniformly, since $\hat{h}\in l^{1}(\widehat{Z})$.
Estimating $r(x)$, we get
\begin{align*}
 \Bigl| \sum_{\chi \in S_{2}} \hat{h}(\chi)\chi(x)\Bigr| & \leq \sum_{\chi \in S_{2}} |\hat{f}(\chi)||\hat{g}(\chi)|^{2}\\
  &\leq \tfrac{1}{4}\delta^{3}\sum_{\chi\in S_{2}} |\hat{g}(\chi)|^{2}  && \text{by definition of $S_{2}$}\\
  &\leq \tfrac{1}{4}\delta^{3}\sum_{\chi\in \widehat{Z}} |\hat{g}(\chi)|^{2}\\
  &= \tfrac{1}{4} \delta^{3} \int |g|^{2} \,dm && \text{by (\ref{eqnIsometry})}\\
  &\leq \tfrac{1}{4} \delta^{3} \int |g| \,dm && \text{since } 0\leq g \leq 1\\
  &\leq \tfrac{1}{4}\delta^{3}\cdot \delta,
\end{align*} so
\begin{equation}\label{eqnRemainderEstimate}
  |r(x)|\leq \tfrac{1}{4}\delta^{4} \text{ for all } x.
\end{equation} It follows that $h(x)$ is positive whenever  $\Re p(x)> \frac{1}{4}\delta^{4}$, meaning the real part of $q(x):= p(x)-\frac{1}{4}\delta^{4}$ is positive.  Choose an $a_{0}$ so that $h(a_{0})\geq \delta^{4}$ and $f(a_{0})>0$. Inequality (\ref{eqnRemainderEstimate}) implies $\Re p(a_{0})\geq \frac{3}{4}\delta^{4}$, so $\Re q(a_{0})\geq \frac{1}{2}\delta^{4}$.  By Lemma \ref{lemBohrPoly}, the set $\{x:\Re q(x)>0\}$ contains a Bohr-$(|S_{1}|, c/(2\pi|S_{1}|))$ neighborhood around $a_{0}$, where $c = \Re q(a_{0}) \geq \frac{1}{2}\delta^{4}$.  This Bohr neighborhood is contained in $\{x: h(x)>0\}$, so we have proved the lemma.

It remains to prove inequalities (\ref{eqnS1Bound}) and (\ref{eqnhBound}). To prove Inequality (\ref{eqnS1Bound}), consider
\begin{align*}
|S_{1}|\cdot \tfrac{1}{16}\delta^{6}&\leq \sum_{\chi \in S_{1}} |\hat{f}(\chi)|^{2}\\
& \leq \sum_{\chi \in \widehat {Z}} |\hat{f}(\chi)|^{2} \\
&= \int |f|^{2} \,dm && \text{by (\ref{eqnIsometry})} \\
 & \leq \int |f| \,dm && \text{since $0\leq f \leq 1$}\\
 & = \delta,
\end{align*} so $|S_{1}|\leq 16\delta^{-5}$.

To prove Inequality (\ref{eqnhBound}), consider
\begin{align*}
\int h\cdot f \,dm
  &= \int h\cdot \bar{f} \,dm\\
  &= \sum_{\chi \in \widehat{Z}} \hat{h}(\chi)\overline{\hat{f}(\chi)} && \text{by (\ref{eqnPlancherel})}\\
 &= \sum_{\chi \in \widehat{Z}} |\hat{f}(\chi)|^{2}|\hat{g}(\chi)|^{2}\\
 &= \sum_{\chi \in \widehat{Z}} |\widehat{f*g}(\chi)|^{2}&&  \text{by (\ref{eqnConvolution})}\\
 &= \int |f*g|^{2} \,dm && \text{by (\ref{eqnIsometry})}\\
 &\geq \bigl|\int f*g \,dm\bigr|^{2} &&  \text{by convexity of $t\mapsto t^{2}$}\\
 & = \delta^{4},
\end{align*}
 so that $\int h\cdot f \,dm\geq \delta^{4}$.  Let $A:=\{a\in Z: f(a)>0\}$.  Observe that $\sup_{a\in A} h(a)  \geq \int h\cdot f \,dm$, and equality holds only if $h$ is constant, so we conclude that $h(a)\geq \delta^{4}$ for some $a\in A$.  \end{proof}

\section{From countable to compact}\label{secReduction}

\begin{proposition}\label{propComplete}
  Let $\Gamma$ be a countable abelian group and let $A, B\subseteq \Gamma$ have $d^{*}(A)>0$, $d^{*}(B)>0$.  Then there are:

   \begin{enumerate}
   \item[$\bullet$] a compact abelian group $Z$ with normalized Haar measure $m$,
   \item[$\bullet$] a homomorphism $\rho:\Gamma\to Z$ such that $\rho(\Gamma)$ is dense in $Z$,
   \item[$\bullet$] Borel functions $f, g: Z\to [0,1]$ satisfying $\int f \,dm=d^{*}(A)$, $\int g \,dm=d^{*}(B)$,
       \end{enumerate}
  such that
  \begin{enumerate}
    \item[$\bullet$]  $f$ is supported on $\overline{\rho(A)}$,
    \item[$\bullet$] $\{\gamma\in \Gamma: f*g*g_{-}(\rho(\gamma))>0\}\subseteq A+B-B$,
  \end{enumerate}
   where $\overline{\rho(A)}$ denotes the topological closure of $\rho(A)$ in $Z$.
\end{proposition}

We now prove Theorem \ref{thmMain} as a consequence of Proposition \ref{propComplete} and Lemma \ref{lemConvolution}.

\begin{proof}[Proof of Theorem \ref{thmMain}.]  Let $A, B \subseteq \Gamma$ have $d^{*}(A), d^{*}(B)>0$.

Set $\delta := \min(d^{*}(A),d^{*}(B))$, and let  $Z, \rho, f,$ and $g$ be as in Proposition \ref{propComplete}.  Let
 \begin{align*}
   h&:=f*g*g_{-}\\
   \tilde{A}&:=\{z\in Z: f(z)>0\}\\
   U&:=\{z\in Z: h(z)>0\},
 \end{align*} so that $\tilde{A}\subseteq \overline{\rho(A)}$ and $\rho^{-1}(U)\subseteq A+B-B$. By Lemma \ref{lemConvolution}, there is an $\tilde{a}\in \tilde{A}$ such that $U-\tilde{a}$ contains a Bohr-$(k,\eta)$ set, where $k, \eta$ depend only on $\delta$.  Lemma \ref{lemBohrIsOpen} provides a neighborhood $V$ of $\tilde{a}$ such that $U-z$ contains a Bohr-$(k,\eta/2)$ set for all $z\in V$.  Since $\tilde{a}\in \overline{\rho(A)}$, there is an $a\in A$ such that $\rho(a)\in V$.  For such $a$, Lemma \ref{lemCompactToCountable} implies $\rho^{-1}(U)-a$ is a Bohr-$(k,\eta/2)$ set contained in $A+B-B-a$.
\end{proof}

\section{Correspondence principle and  proof of Proposition \ref{propComplete}}\label{secErgodic}
In this section we fix a countable abelian group $\Gamma$.  We will exploit the theory measure preserving actions of $\Gamma$, see \cite{EinsiedlerWard}, \cite{Furstenberg}, or \cite{Glasner} for general references, and \cite{GrIsr} or \cite{GrAdv} for similar applications.

\subsection{Measure preserving systems}

 A \emph{measure preserving $\Gamma$-system} (or briefly, \emph{$\Gamma$-system}) is a quadruple $(X,\mathscr X,\mu,T)$, where $(X,\mathscr X,\mu)$ is a probability measure space and $T$ is an action of $\Gamma$ on $X$ preserving $\mu$:
 \begin{equation}\label{eqnPreserveMeasure}
   \mu(T^{-\gamma} D)=\mu(D)
 \end{equation} for all measurable $D\subseteq X$ and all $\gamma\in \Gamma$.  Note that Equation (\ref{eqnPreserveMeasure}) yields the identities
 \begin{align}
   \mu(C\cap T^{-\gamma}D) &= \mu(T^{a}(C\cap T^{-\gamma}D))=\mu(T^{a}C \cap T^{a-\gamma}D),\\
   \label{eqnL2Preserved} \int f\cdot g\circ T^{\gamma} \,d\mu &= \int f\circ T^{-\gamma}\cdot g \,d\mu,
 \end{align}
 for all measurable $C, D\subseteq X$, all $f, g\in L^{2}(\mu)$, and all $a, \gamma\in \Gamma$.

A $\Gamma$-system is \emph{ergodic} if for every $D\subseteq X$ satisfying $\mu(D \triangle T^{\gamma}D)=0$ for all $\gamma\in \Gamma$, we have $\mu(D)=0$ or $\mu(D)=1$.

A \emph{factor} of a $\Gamma$-system $(X,\mathscr X,\mu,T)$ is a $\Gamma$-system $(Y,\mathscr Y,\nu,S)$ together with a \emph{factor map} $\pi:X\to Y$, defined for $\mu$-a.e. $x\in X$, such that
\begin{equation}\label{eqnEquivariant}
\pi(T^{\gamma}x)=S^{\gamma}\pi(x) \ \text{ for $\mu$-a.e.\ $x\in X$ and all $\gamma\in \Gamma$}.
\end{equation}
The space $L^{2}(\nu)$ may be identified with the subspace of $L^{2}(\mu)$ consisting of functions of the form $g\circ \pi$, where $g\in L^{2}(\nu)$.  The $\sigma$-algebra $\pi^{-1}(\mathscr Y)$ consists of those sets which are $\mu$-a.e.\ equal to a set of the form $\pi^{-1}(C)$, where $C\subseteq Y$.  Note that $L^{2}(\nu)$ may be identified with those elements of $L^{2}(\mu)$ which are $\pi^{-1}(\mathscr Y)$-measurable.  Let $P_{\mathscr Y}:L^{2}(\mu)\to L^{2}(\mu)$ denote the orthogonal projection onto the space of $\pi^{-1}(\mathscr Y)$-measurable functions.

 If $g\in L^{2}(\mu)$, let $\tilde{g}$ be the element of $L^{2}(\nu)$ satisfying $P_{\mathscr Y}g = \tilde{g}\circ \pi$. Note that if $f$ is $\pi^{-1}(\mathscr Y)$-measurable, we have
\begin{equation}\label{eqnFactorIntegrate}
  \int f\cdot g\circ T^{\gamma} \,d\mu = \int  \tilde{f}\cdot \tilde{g}\circ S^{\gamma} \,d\nu \ \ \text{ for all } \gamma\in \Gamma,
\end{equation}
where $\tilde{f}\in L^{2}(\nu)$ satisfies $f= \tilde{f}\circ \pi$.
\subsection{Group rotations}
A \emph{group rotation} is a $\Gamma$-system $(Z,\mathcal Z, m,R_{\rho})$, where $Z$ is a compact abelian group with normalized Haar measure $m$, $\rho: \Gamma\to Z$ is a homomorphism, and the action $R_{\rho}$ is given by $R_{\rho}^{\gamma}(z) = z+\rho(\gamma)$.  The group rotation $(Z,\mathcal Z, m,R_{\rho})$ is ergodic iff $\rho(\Gamma)$ is dense in $Z$.

The \emph{Kronecker factor} of  a $\Gamma$-system $\mathbf X$ is the maximal factor $\mathbf Y$ of $\mathbf X$ such that $\mathbf Y$ is isomorphic to a group rotation. When $\mathbf X$ is ergodic, such a factor always exists and is ergodic (although it may be trivial).  See \cite{EinsiedlerWard}, \cite{Furstenberg}, or \cite{Glasner} for the existence of the Kronecker factor and its properties.

\subsection{A correspondence principle}

The following lemma is standard, but we outline the proof for completeness.
\begin{lemma}\label{lemCorrespondence1}
  If $B\subseteq \Gamma$ has $d^{*}(B)>0$, there is an ergodic $\Gamma$-system $(X,\mathscr X,\mu,T)$ and a $D\subseteq X$ having $\mu(D)\geq d^{*}(B)$ such that $B-B$ contains $R(D):=\{\gamma: \mu(D\cap T^{\gamma}D)>0\}$.
\end{lemma}

\begin{proof}
  Let $\Omega = \{0,1\}^{\Gamma}$ with the product topology, so that $\Omega$ is a compact metrizable space.  Let $T$ be the action of $\Gamma$ on $\Omega$ defined by $(T^{\beta}\omega)(\gamma):=\omega(\gamma+\beta)$. Let $\mathbf \Phi=(\Phi_{n})_{n\in \mathbb N}$ be a F{\o}lner sequence for $\Gamma$ such that $d_{\mathbf \Phi}(B)=d^{*}(B)$.  Consider $x:=1_{B}\in\Omega$, and let $X$ be the orbit closure of $x$ under $T$, meaning $X$ is the closure of the set $\{T^{\gamma}x: \gamma\in \Gamma\}$.  Let $D$ be the set $\{\omega\in X: \omega(0)=1\}$, so that $D$ is a clopen subset of $X$.  Note that $B=\{\gamma: T^{\gamma}x\in D\}$.

  We will find a $T$-invariant measure $\mu$ on $X$ such that $\mu(D)=d^{*}(B)$.  Let $\delta_{x}$ be the Dirac mass concentrated at $x$, and for each $n$, let $\nu_{n}:= \frac{1}{|\Phi_{n}|}\sum_{\gamma\in \Phi_{n}}\delta_{T^{\gamma}x}$. Let $\nu$ be a $\text{weak}^{*}$ limit of the $\nu_{n}$.  Then $\nu$ is a $T$-invariant probability measure, while
  \[\nu(D) = \lim_{n\to \infty} \frac{1}{|\Phi_{n}|}|\{\gamma\in \Phi_{n}: T^{\gamma}x\in D\}| = \lim_{n\to \infty} \frac{|B\cap \Phi_{n}|}{|\Phi_{n}|} = d^{*}(B).\]
Applying ergodic decomposition (\cite{EinsiedlerWard}, Theorem 8.20), we may find an ergodic $T$-invariant measure $\mu$ on $X$ such that $\mu(D)\geq \nu(D)$.

We now show that $\mu(D\cap T^{\gamma}D)>0$ implies $\gamma\in B-B$.  In fact, if $D\cap T^{\gamma}D\neq \varnothing$, then there exists $y\in D$ such that $T^{\gamma}y\in D$.  Since $y$ is a limit of points of the form $T^{\beta}x$ and $D$ is open, there exist $\beta\in \Gamma$ such that $T^{\beta}x\in D$ and $T^{\gamma}T^{\beta}x\in D$, which implies $\beta\in B$ and $\gamma+\beta\in B$. We then have $\gamma\in B-B$.
\end{proof}

If $D\subseteq X$ and $A\subseteq \Gamma$, let $A\cdot D:= \bigcup_{a\in A}T^{a}D$.

\begin{lemma}\label{lemCorrespondence2} If $\mathbf X$ is the $\Gamma$-system obtained in Lemma \ref{lemCorrespondence1} and $A\subseteq \Gamma$, then $A+B-B$ contains $\{\gamma\in \Gamma: \mu((A\cdot D)\cap T^{\gamma}D)>0\}$.
\end{lemma}

\begin{proof}
Note that $\mu(T^{a}D\cap T^{\gamma}D)>0$ iff $\mu(D\cap T^{\gamma-a}D)>0$, which implies $\gamma-a \in B-B$, meaning $\gamma\in a+B-B$.  It follows that $\mu((A\cdot D) \cap T^{\gamma}D)>0$ implies $\gamma\in A+B-B$.
\end{proof}

\subsection{Proof of Proposition \ref{propComplete}}

From now on we fix:

\begin{enumerate}
\item[$\bullet$] $A$, $B\subseteq \Gamma$ having $d^{*}(A)>0$, $d^{*}(B)>0$,

\item[$\bullet$] a F{\o}lner sequence $\mathbf \Phi$ satisfying $d_{\mathbf \Phi}(A)=d^{*}(A)$,

\item[$\bullet$]  an ergodic $\Gamma$-system $\mathbf X= (X,\mathscr X,\mu,T)$ and a set $D\subseteq X$ satisfying the conclusion of Lemma \ref{lemCorrespondence2},

\item[$\bullet$]  the Kronecker factor $\mathbf Z=(Z,\mathcal Z, m,R_{\rho})$ of $(X,\mathscr X,\mu,T)$, with factor map $\pi:X\to Z$.  Note that $\rho:\Gamma\to Z$ is a homomorphism with $\rho(\Gamma)$ dense in $Z$.

\item[$\bullet$]  a function $g:Z\to [0,1]$ satisfying $g\circ \pi = P_{\mathcal Z}1_{D}$ $\mu$-almost everywhere.

\end{enumerate}

Here $P_{\mathcal Z}:L^{2}(\mu)\to L^{2}(\mu)$ is the orthogonal projection onto the space of $\pi^{-1}(\mathcal Z)$-measurable functions.  As a special case of Equation (\ref{eqnEquivariant}), we get
\begin{align}\label{eqnKroneckerFactor}
g(\pi(x)+\rho(\gamma))=P_{\mathcal Z}1_{D}(T^{\gamma}x) \text{ for $\mu$-a.e. } x \text{ and all } \gamma\in \Gamma.
\end{align}
Furthermore, if $F$ is $\pi^{-1}(\mathcal Z)$-measurable, meaning $F=\tilde{F}\circ \pi$ for some $\tilde{F}\in L^{2}(m)$, then Equation (\ref{eqnFactorIntegrate}) implies
\begin{equation}\label{eqnKroneckerIntegrate}
  \int F\cdot 1_{D}\circ T^{-\gamma} \,d\mu = \int \tilde{F} \cdot g\circ R_{\rho}^{-\gamma}\, dm = \int \tilde{F}(z) \cdot g(z-\rho(\gamma)) \,dm(z)
\end{equation}
for all $\gamma\in \Gamma$.

The lemmas and proofs in the remainder of this section will refer to the objects defined above.  Our goal, in light of Lemma \ref{lemCorrespondence2}, is to describe those $\gamma$ for which $\mu((A\cdot D)\cap T^{-\gamma}D)>0$.  Our approach is similar to the proofs of Proposition 3.2 of \cite{GrIsr} and Proposition 4.2 of \cite{GrAdv}.

In order to understand $A\cdot D$, we consider averages of functions supported on $A\cdot D$.  We have fixed a F{\o}lner sequence $\mathbf \Phi=(\Phi_{n})_{n\in \mathbb N}$ such that $d_{\mathbf \Phi}(A)=d^{*}(A)$.  Consider the sets $A_{n}:=A\cap \Phi_{n}$.  Observe that $1_{D}\circ T^{-a}$ is supported on $A\cdot D$ for each $a\in A$, so each average
\begin{equation}\label{eqnDefAvg}
F_{n}:=\frac{1}{|\Phi_{n}|}\sum_{a\in A_{n}} 1_{D}\circ T^{-a}
\end{equation}
is supported on $A\cdot D$, and every weak $L^{2}(\mu)$ limit of these averages is also supported on $A\cdot D$.  Passing to a subsequence of $\mathbf \Phi$, we may assume that the limit of $(F_{n})_{n\in \mathbb N}$ exists. Analyzing this limit will lead to the following lemma.  See Section \ref{secFourier} for the definitions of $f*g$ and $g_{-}.$

\begin{lemma}\label{lemRecurrenceToSumset}
\begin{enumerate}
   \item[(i)] The set $A\cdot D$ supports a function of the form $f*g\circ \pi$, where $f:Z\to [0,1]$ has $\int f \,dm = d^{*}(A)$, $f$ is supported on $\overline{\rho(A)}$, and $g$ is defined above.
       \smallskip
   \item[(ii)] $\{\gamma\in \Gamma: \mu((A\cdot D)\cap T^{\gamma}D)>0\}\supseteq \{\gamma: f*g*g_{-}(\rho(\gamma))>0\}$.

\end{enumerate}

\end{lemma}

\begin{proof}  To prove (i),  let $F$ be a weak $L^{2}(\mu)$ limit of the sequence $(F_{n})_{n\in \mathbb N}$, defined in Equation (\ref{eqnDefAvg}).  Then $F$ is supported on $A\cdot D$, and by Lemma \ref{lemWeakLimits}, $F$ has the form $f*g\circ \pi$, where $f$ is as described in Part (ii) of that lemma.

Proof of Part (ii).  We must prove the implication
\begin{equation}\label{eqnTheImplication}
  f*g*g_{-}(\rho(\gamma))>0 \implies \mu((A\cdot D)\cap T^{\gamma}D)>0.
\end{equation} Suppose $\gamma\in \Gamma$ satisfies
\begin{equation}\label{eqnIntersectionIntegral}
  \int (f*g\circ \pi) \cdot 1_{D}\circ T^{-\gamma} \,d\mu>0.
\end{equation}  Equation (\ref{eqnKroneckerIntegrate}) implies the integral in Inequality (\ref{eqnIntersectionIntegral}) is positive if and only if $\int f*g(z)\cdot g(z-\rho(\gamma))  \,dm(z)>0,$ meaning $f*g*g_{-}(\rho(\gamma))>0$.  By Part (i), inequality (\ref{eqnIntersectionIntegral}) implies $\mu((A\cdot D)\cap T^{\gamma}D)>0$, so we have proven the implication (\ref{eqnTheImplication}).
   \end{proof}

\begin{proof}[Proof of Proposition \ref{propComplete}]
Proposition \ref{propComplete} now follows from Lemmas \ref{lemCorrespondence2} and \ref{lemRecurrenceToSumset}.
\end{proof}

\begin{lemma}\label{lemWeakLimits}
If the weak $L^2(\mu)$ limit $F$ of the sequence \[F_n:=\frac{1}{|\Phi_n|}\sum_{a\in A_{n}} 1_{D}\circ T^{-a}\] exists and $g:Z\to [0,1]$ is the function satisfying Equation \textup{(\ref{eqnKroneckerFactor})}, then

\begin{enumerate}
\item[(i)]
$F$ is $\pi^{-1}(\mathcal Z)$-measurable, and
\begin{align*}
  F = \lim_{n\to \infty} \frac{1}{|\Phi_{n}|} \sum_{a \in A_{n}} g\circ R_{\rho}^{-1}\circ \pi \text{ weakly in } L^{2}(\mu).
\end{align*}

\smallskip

\item[(ii)] $F=(f*g) \circ \pi$, where $f:Z\to [0,1]$ satisfies $\int f \,dm = d^{*}(A)$ and $f$ is supported on $\overline{\rho(A)}$. \end{enumerate}
\end{lemma}
\begin{proof}
  Part (i) is a consequence of the proof of Corollary 2.7 of \cite{GrAdv}. To prove Part (ii), it suffices to show that
\begin{equation}\label{eqnInnerProd}
    \int F\cdot \psi \,d\mu= \int  (f*g)\circ \pi \cdot \psi  \,d\mu \text{\ \ for every } \psi\in L^{2}(\mu).
  \end{equation}Part (i) already proves Equation (\ref{eqnInnerProd}) for those $\psi$ orthogonal to the $\pi^{-1}(\mathcal Z)$-measurable functions, so we may assume that $\psi$ is $\pi^{-1}(\mathcal Z)$-measurable, meaning $\psi = \tilde{\psi}\circ \pi$ for some $\tilde{\psi} \in L^{2}(m)$.  It suffices to establish the identity (\ref{eqnInnerProd}) for an $L^{2}(\mu)$-dense set of functions $\psi$, so we may assume that $\psi=\tilde{\psi}\circ \pi$, where $\tilde{\psi}:Z\to \mathbb C$ is continuous.  Let $F$ be a weak $L^{2}(\mu)$ limit of the $(F_{n})_{n\in \mathbb N}$, and let $\nu$ be a measure as in the conclusion of Lemma \ref{lemWeakStar}, so that $d\nu =  f \,dm$ for some $f$ as in the conclusion of the  lemma. Then we apply Part (i) to compute $F$:
  \begin{align}
    \int F \cdot \psi \,d\mu &= \lim_{n\to \infty} \frac{1}{|\Phi_{n}|} \sum_{a\in A_{n}} \int g(\pi(x)-\rho(a)) \cdot \tilde{\psi}(\pi(x)) \,d\mu(x)\\
    &= \lim_{n\to \infty} \frac{1}{|\Phi_{n}|}\sum_{a\in A_{n}}  \int g(\pi(x)) \cdot \tilde{\psi}(\pi(x)+\rho(a))\,d\mu(x)\\
    &= \lim_{n\to \infty} \int  g(z) \cdot \frac{1}{|\Phi_{n}|}\sum_{a\in A_{n}} \tilde{\psi}(z+\rho(a)) \,dm(z). \label{eqnLast}
    \end{align}
We used Equations (\ref{eqnL2Preserved}) and (\ref{eqnKroneckerIntegrate}) to get the second and third lines. Applying Lemma \ref{lemWeakStar} to evaluate the limit in Equation (\ref{eqnLast}), we  find
\[\lim_{n\to \infty} \frac{1}{|\Phi_{n}|} \sum_{a\in A_{n}}\tilde{\psi}(z+\rho(a)) = \int \tilde{\psi}(z+w)f(w) \,dm(w) \ \text{ for all } z\in Z,
\]  where $f$ is as described above. We then evaluate, continuing from Equation (\ref{eqnLast}),
   \begin{align*}
   \int F\cdot \psi \,d\mu &= \int g(z) \cdot \int \tilde{\psi}(z+w)f(w) \,dm(w) \,dm(z)\\
    &= \int \int f(w)g(z-w) \,dm(w)\, \tilde{\psi}(z)\, dm(z)\\
    &= \int  f*g(z) \cdot \tilde{\psi}(z) \,dm(z)\\
    &= \int  (f*g)\circ \pi \cdot \psi \,d\mu,
  \end{align*}
establishing Equation (\ref{eqnInnerProd}).
The second line uses Fubini and translation invariance of $dm$, the third is just the definition of $f*g$, and the last line is the special case of Equation (\ref{eqnFactorIntegrate}) with $\gamma=0$.
\end{proof}

\begin{lemma}\label{lemWeakStar}  Consider the Borel measures $\nu_{n}$ on $Z$ given by
\begin{align*}
  \int \psi \,d\nu_{n}= \frac{1}{|\Phi_{n}|} \sum_{\gamma\in A_{n}} \psi(\rho(\gamma))  && \text{ for } \psi:Z\to \mathbb C.
\end{align*}
   Let $\nu$ be a $\text{weak}^{*}$ limit of the $\nu_{n}$.  Then $\nu$ is absolutely continuous with respect to Haar measure $m$, and
  \begin{enumerate}
     \item[(i)]   $d\nu = f\, dm$, where $f: Z\to [0,1]$,
     \item[(ii)] $\nu(Z)=d^{*}(A)$,
     \item[(iii)] $f$ is supported on $\overline{\rho(A)}$.
  \end{enumerate}
\end{lemma}

\begin{proof}  Parts (i) and (ii) are established in the proof of Lemma 2.11 of \cite{GrAdv}. Part (iii) follows from the fact that $\int \psi \,d\nu=0$ whenever $\psi: Z\to \mathbb C$ is a continuous function which vanishes on $\overline{\rho(A)}$. \end{proof}

\bibliographystyle{amsplain}
\bibliography{CDbib}

\end{document}